\documentclass[10pt]{article}

\usepackage{amsmath}
\usepackage{times}
\usepackage{array, amssymb, amsmath,graphics}
\newtheorem{remark}{Remark}
\newtheorem{problem}{Problem}
\newtheorem{theorem}{Theorem}

\newtheorem{corollary}{Corollary}

\newenvironment{proof}{{\bf Proof.}}{\hspace*{1mm}\hfill\rule{2mm}{2mm}}

\newtheorem{pretheorema}{{\bf Theorem}}

\newenvironment{theorema}[1]{\begin{pretheorema}
{\hspace{-0.2em}{\rm #1}{\bf}}}{\end{pretheorema}}
\newtheorem{prelem}{{\bf Theorem}}

\newtheorem{prelemlem}{{\bf Lemma}}

\def\m#1#2{\raise 0.2ex\hbox{
    ${#1_{\displaystyle #2}}$}}
\def\x#1{\raise 0.5ex\hbox{
    ${#1}$}}
\def\n#1{\vbox to 3mm{\vspace{0mm}\vfill \hbox to 4.5mm{\hfill
             $#1$\hfill} \vfill }}

\def\x{{\bf x}}

\newcommand{\arb}[1]{{\rm arb}(#1)} 
\title{\bf On the construction of tree decompositions of hypercubes}
\author{{\sc Negin Karisani,} {\sc  E.S.~Mahmoodian\thanks{The research of this author is partially supported by a grant from the INSF.}}
 }
\date{}
\begin{document}
\maketitle
\vspace*{-1cm}
\begin{center}
\footnotesize{
Department of Mathematical Sciences \\
 Sharif University of Technology \\
 P.O. Box 11155--9415 \\
 Tehran, I.R. IRAN \\
 karisani@alum.sharif.edu \\ emahmood@sharif.edu}
 \end{center}
%
\begin{abstract}
%
%
There are different concepts regarding to tree decomposition of a graph $G$. For the Hypercube $Q_n$,
these concepts have been shown to have many applications. But some diverse papers on this subject make it difficult to follow
 what is precisely known.  In this note first we will mention some known results on the tree decomposition of
  hypercubes and then introduce new  explicit
constructions for the previously known and unknown cases.
\end{abstract}
%
\section{Introduction and preliminaries}       
\label{1}




There is much interest in the theoretical study of the structure of hypercubes because their structure has played an important role in the development of parallel processing and is still quite popular and influential.
As a graph the \textit{$n$-cube} or \textit{$n$-dimensional hypercube} $Q_n$ is a graph  in which the vertices are all binary vectors of length $n$,
and two vertices are adjacent if and only if the Hamming distance between
them is 1, i.e. their components differ in one place. There are many challenging conjectures about hypercubes.
 $Q_n$ is also defined recursively in terms of the cartesian product of two graphs \cite{MR949280} as follows:
$$\begin{array}{l}
Q_1=K_2 \\
Q_n=K_2 \square Q_{n-1}.
\end{array}$$

There are different concepts regarding to tree decomposition of a graph $G$. These concepts have been shown to have many applications including for the hypercube, $Q_n$ (see~\cite{MR2817074}, \cite{parhami1999introduction} and \cite{wang1998tree}).
Here we mention some of these concepts as follows.
The {\it arboricity}, $\arb{G}$, of a graph $G$ is the minimum number of edge-disjoint forests into which $G$ can be decomposed. A family of subtrees of a graph $G$ whose edge sets form a partition of the edge set of $G$ is called \textit{tree decomposition} of $G$. The minimum number of trees in a tree decomposition of $G$ is called the \textit{ tree number} of $G$ and denoted by $\tau (G)$. Since each forest on $|V|$ vertices has at most $|V|-1$ edges, $\lceil |E(G)|/(|V(G)|-1) \rceil$ is a trivial lower bound for both the arboricity and the tree number of $G$. The arboricity is clearly a lower bound for tree number. Another related decomposition parameter for a graph $G$ is the \textit{spanning tree packing number} denoted by $\sigma(G)$, is the maximum number of edge disjoint spanning trees contained in $G$. In an analogous way for a graph $G$, $\lfloor |E(G)|/(|V(G)|-1) \rfloor$ is a trivial upper bound for $\sigma(G)$.
So we have the following inequalities:
\begin{equation}\label{inq}
\sigma(G) \leq \Big\lfloor \frac{|E(G)|}{|V(G)|-1} \Big\rfloor \leq  \Big\lceil \frac{|E(G)|}{|V(G)|-1} \Big\rceil \le \arb{G} \leq \tau(G).
\end{equation}

There are quite few diverse papers on this subject about $Q_n$ which makes it difficult to follow what is precisely known. For example, recently in a note Wagner and Wild~\cite{Wagner20121819} stated  that ``For even $n$, an explicit construction of $n/2$ edge-disjoint spanning trees in  $Q_n$
appears in~\cite{MR1695754}. These trees are not isomorphic. There are $n/2$ leftover edges, which form a path. For odd $n$, an explicit
construction has not yet been found''. In this note we summarize some known results on the tree decomposition of hypercubes and
introduce explicit
constructions for the previously known and unknown cases.

\section{Tree packing number}
The motivation of finding the maximum number of edge-disjoint spanning trees arises in the context of developing efficient routing algorithms such as wormhole routing in distributed memory parallel multicomputers \cite{wang1998tree}.

Tutte $(1961)$ and Nash-Williams $(1961)$ independently proved the same sufficient conditions for the existence of $k$ edge-disjoint spanning trees. The following corollary can be deduced from their result:
\begin{corollary}
\label{tpkn}
(\cite{MR2368647}~page 572) Every $2k$-edge-connected graph contains $k$ edge-disjoint spanning trees.
\end{corollary}
From Corollary~\ref{tpkn}, it is easily deduced that $\sigma(Q_n)= \lfloor n/2 \rfloor$, also in \cite{MR1812338} by using an observation of Paul Catlin, the spanning tree packing number  for several family of graphs including hypercubes, is determined.

\begin{corollary}
\label{tree-packing}
The spanning tree packing number of $Q_n$ is $\lfloor n/2 \rfloor$.
\end{corollary}
\begin{proof}
Since the hypercube $Q_n$ is $n$-connected, by Corollary~\ref{tpkn} $\sigma(Q_n) \geq \lfloor n/2 \rfloor$. On the other hand by  the upper bound that was mentioned in the previous section we have $\sigma(Q_n) \leq \lfloor{{n2^{n-1}}/ {2^n-1}}\rfloor = \lfloor n / 2 \rfloor$. So $\sigma(Q_n) = \lfloor n/2 \rfloor$.
\end{proof}

The number of uncovered edges in the tree packing of Corollary~\ref{tree-packing}, is $k$ if $n=2k$, and is $2^{2k}+k$  if $n=2k+1$. The structure of uncovered edges is not clear in these cases, but Catlin \cite{MR1163250} has a result similar to Corollary~\ref{tpkn} which shows that \textit{any} set of $k$ edges in $G$ can be the set of leftover edges.
\begin{theorema} (\cite{MR1163250}, page 182) \label{ctl}
Let $G$ be a graph and let $k \in \mathbb{N} $. Then $G$ is $2k$-edge-connected if and only if for any set $E_k$ of $k$ edges of $G$, the subgraph $G \setminus E_k$ has at least $k$ edge-disjoint spanning trees.
\end{theorema}
In the following theorem Barden et al.   give a construction of edge-disjoint spanning trees for $Q_{2k}$.
\begin{theorema}(\cite{MR1695754}, page 15) \label{eds}
Let $Q_{2k}$ denote a hypercube of dimension $2k$. Then $k$ edge-disjoint spanning trees can be embedded in $Q_{2k}$, with the remaining $k$ edges forming a path.
\end{theorema}
Wagner and Wild in~\cite{Wagner20121819} referred to Theorem~\ref{eds} and stated  that: ``For odd $n$, an explicit construction has not yet been found''. In the next section we give a construction for such $n$.

\section{Arboricity and tree number}
The arboricity of a graph is a measure of how dense the graph is. This invariant plays a significant role in bounding the complexity of certain algorithms, some of those are presented in~\cite{MR2817074}.

An exact formula for arboricity is due to Nash-Williams $(1964)$:
\begin{theorema} (\cite{MR2817074}, page 240) \label{Tnash}
For any graph $G$ with at least one edge,
$$\arb{G}=\mathop{\rm max}_{H\subseteq G} \Big\lceil {|E(H)| \over |V(H)|-1} \Big\rceil,$$
where the maximum is taken over all nontrivial subgraphs  $H\subseteq G$.
\end{theorema}
In \cite{MR1154585} it is shown that arboricity for general graphs can be calculated by a polynomial-time algorithm but in this section we consider simple constructions for hypercubes.

In the following theorem, Truszczy{\'n}ski $(1988)$ shows an exact value of a tree number for hypercubes.
\begin{theorema}(\cite{MR984777}, page 279) \label{truz}
If $G=C_{m_1} \square C_{m_2} \square \cdots \square C_{m_n}$  then $\tau(G)=n+1$. If $G=P_{m_1} \square C_{m_2} \square \cdots \square C_{m_n}$  then $\tau(G)=n$. In particular, $\tau(Q_n)=\lceil {n+1\over 2}\rceil$.
\end{theorema}
The following corollary follows easily.
\begin{corollary}
The arboricity of hypercube $Q_n$ is $\lfloor n/2 \rfloor +1$. 
\end{corollary}
\begin{proof}
From Inequalities~(\ref{inq}) and Theorem~\ref{truz} we have $\arb{Q_n} \leq \lceil {n+1 \over 2}\rceil=\lfloor {n \over 2} \rfloor +1$. So it suffices to prove the other side.
Since each forest of $Q_n$ has at most  $2^n-1$ edges, so by Theorem~\ref{Tnash} we have $\mbox{arb}(Q_n) \geq \lceil{{n2^{n-1}} \over {2^n-1}}\rceil = \lceil {n \over 2}(1+{1 \over 2^n-1}) \rceil \geq \lfloor{ n \over 2} \rfloor + \varepsilon$, where $\varepsilon > 0$. Thus $\mbox{arb}(Q_n) \geq \lfloor {n \over 2} \rfloor +1$.
\end{proof}

The construction in the proof of  Theorem~\ref{truz} for the arboricity of hypercube, gives a decomposition of edges into trees where only one of them is a spanning tree. Here we will introduce a construction with the maximum possible number of spanning trees.

For  $ Q_{2k} $, Theorem~\ref{eds} gives a construction for arboricity with $\lfloor n/2 \rfloor$ spanning trees and a path of length $k$. In the next theorem by a similar method we consider a new construction for $Q_{2k}$ which can be extended to  a construction for $Q_{2k+1}$, with maximum possible number of spanning trees and also a construction for its arboricity.
\begin{theorem}
\label{t2k}
The hypercube  $ Q_{2k} $ may be decomposed into $k$ edge-disjoint spanning trees and a matching of size  $ k $.
 \end{theorem}
\begin{proof}
By induction on $k$. For $k=1$, the $2$-dimensional hypercube $Q_2$, contains one spanning tree and a matching of size $1$. Assume the theorem holds for every even number $n \leq 2k$, we construct a decomposition of $2(k+1)$-dimensional hypercube, $Q_{2(k+1)}$, with $k+1$ spanning trees such that the remaining edges make a matching of size $k+1$. Since  $ Q_{2(k+1)} = Q_2 \square Q_{2k}$,  it can be decomposed into four $2k$-dimensional hypercubes, \ $ Q_{2k}^ {(i)}$, \ $ 1\leq i \leq4$, such that the pairs $\{Q_{2k}^{(1)}, Q_{2k}^{(2)}\}$,  $\{Q_{2k}^{(1)},Q_{2k}^{(4)}\}$, $\{Q_{2k}^{(2)},Q_{2k}^{(3)}\}$,  and $\{Q_{2k}^{(3)},Q_{2k}^{(4)}\}$  are connected to each other. For example this connection for $\{Q_{2k}^{(1)} Q_{2k}^{(2)}\}$
is such that every vertex in $Q^{(1)}_{2k}$ is adjacent to its corresponding vertex in $Q^{(2)}_{2k}$. We denote the set of edges between $Q^{(1)}_{2k}$  and $Q^{(2)}_{2k}$ by $M^{(1,2)}$. It is a matching of size   $2^{2k}$. In a similar way  $M^{(2,3)}$, $M^{(3,4)}$ and $M^{(1,4)}$ are defined. By the induction  hypothesis for $ Q^{(1)}_{2k}$, we have a decomposition, say
$$
\left\{ \begin{array}{lr}
{\cal T}^{(1)} = \lbrace T_1^{(1)} , T_2^{(1)} , ... , T_k^{(1)} \rbrace \\
I^{(1)} = \lbrace e_1^{(1)} , e_2^{(1)} , ... , e_k^{(1)} \rbrace,
\end{array}\right.
$$
where ${\cal T}^{(1)}$ is the set of $k$ spanning trees and $I^{(1)}$ is the set of $k$ remaining independent edges. We denote the corresponding decomposition in other $Q^{(i)}_{2k}$'s, $i=2,3,4$, as:
 $$
\left\{ \begin{array}{lr}
{\cal T}^{(i)} = \lbrace T_1^{(i)} , T_2^{(i)} , ... , T_k^{(i)} \rbrace \\
I^{(i)} = \lbrace e_1^{(i)} , e_2^{(i)} , ... , e_k^{(i)} \rbrace.
\end{array}\right.
$$
 So for each $j$, $1 \leq j \leq k$, $e^{(1)}_j$ and $e^{(2)}_j$, are two corresponding edges in $Q^{(1)}_{2k}$ and $Q^{(2)}_{2k}$. They are connected with two edges of $M^{(1,2)}$. We choose one of those two edges, and denotes it by $f^{(1,2)}_j$. Let $F^{(1,2)}$ denote the set of $k$ chosen $f^{(1,2)}_j$'s. In a similar way  $F^{(2,3)}$, $F^{(3,4)}$ and $F^{(1,4)}$ are defined.

First we construct $k-1$ spanning trees as follows. Consider the four spanning trees $T^{(i)}_1$ of $Q_{2k}^{(i)}$, $1~\leq~i~\leq 4$, and connect them by using $f_1^{(1,2)}$, $f_1^{(2,3)}$ and $f_1^{(3,4)}$. It is clear that the resulting subgraph is a spanning tree of  $Q_{2(k+1)}$, let it be $\widehat{T}_1$. In a similar way $\widehat{T}_2, ... ,\widehat{T}_{k-1}$ are constructed.
Now consider the tree $T_k^{(1)}$, and the sets $I^{(2)}$, $I^{(3)}$ and $I^{(4)}$.  Connect them by using   $M^{(1,2)} \setminus F^{(1,2)} $, $M^{(2,3)} \setminus F^{(2,3)}$ and $M^{(3,4)} \setminus F^{(3,4)}$. The resulting subgraph is a spanning tree, let $\widehat{T}_k$  denote this tree.

The last spanning tree, $\widehat{T}_{k+1}$, is constructed by the union of $T_k^{(2)}$, $T_k^{(3)}$, $T_k^{(4)}$, and the set  $M^{(1,4)}$ and two edges, $f_k^{(1,2)}$ and  $f_k^{(3,4)}$. Finally the set $I^{(1)} \cup \{ f_k^{(2,3)} \}$ makes a set of $k+1$ independent edges of $Q_{2(k+1)}$.

To show that the three kinds of edge sets $\widehat{T}_j (1 \leq j \leq k-1), \widehat{T}_k, \widehat{T}_{k+1}$ are indeed spanning trees of  $Q_{2(k+1)}$, it suffices to show that each edge set is incident with all $2^{2k+2}$ vertices and each edge set has cardinality $2^{2k+2}-1$. The first, takes place obviously in all three cases, and for the second we have 

$
|E(\widehat{T}_j )| = 4(2^{2k}-1)+3=2^{2k+2}-1, \hspace*{.2cm}(1 \leq j \leq k-1), \\
|E(\widehat{T}_k)|= (2^{2k}-1)+3(2^{2k}-k)+3k =2^{2k+2}-1, \\ 
|E(\widehat{T}_{k+1})|= 3(2^{2k}-1)+2^{2k}+2=2^{2k+2}-1.$
\end{proof}

\begin{theorem} \label{arbodd}
The hypercube $Q_{2k+1}$ may be decomposed into $k$ edge-disjoint spanning trees with the remaining edges forming a forest with $k$ components.
\end{theorem}
\begin{proof}
The hypercube $Q_{2k+1}$, consists of two copies of $2k$-dimensional hypercubes, $Q_{2k}^{(1)}$ and $Q_{2k}^{(2)}$, with a perfect matching between them. We consider definitions and notations introduced in the proof of Theorem~\ref{t2k}
for the  hypercubes $Q_{2k}^{(i)}$, $i =1,2$,
and also the construction for the decomposition of hypercubes $Q_{2k}$. At first we find $k$ spanning trees for $Q_{2k}^i$'s. We may assume  $Q_{2k}^{(i)}$, $i =1,2$, has the set $\{T_1^{(i)}, \ldots ,T_k^{(i)} \}$ of spanning trees.  Consider the union of trees $T_1^{(1)}$ and $T_1^{(2)}$, and the edge $f_{1}^{(1,2)}$. Clearly this is a spanning tree for $Q_{2k+1}$, let it be $\widehat{T}_1$. In a similar way construct $\widehat{T}_2,\ldots,\widehat{T}_{k-1}$. The spanning tree $\widehat{T}_{k}$ is constructed by the union of $T_k^{(1)}$, $M^{(1,2)} \setminus F^{(1,2)}$, and the set $I^{(2)}$. At last the set $I^{(1)} \cup \{f_k^{(1,2)}\} \cup T_k^{(2)}$ makes a forest with $k$ components.~\end{proof}
\begin{remark}
Note that the {\em  existence} of maximum spanning tree decomposition of $Q_n$ can be deduced from Theorem~\ref{ctl}, but our results are {\em constructive}.
\end{remark}
\begin{remark}
Using the construction of Theorem~\ref{eds}, by the same method as in the proof of Theorem~\ref{arbodd}, a set of $k$ edge-disjoint spanning trees for $Q_{2k+1}$  can be deduced. But the remaining edges will not necessarily be a forest.
\end{remark}
\begin{problem}
It will be interesting to try to find a construction for Catlin's Theorem~\ref{ctl}.
\end{problem}
\section*{Acknowledgements}
The authors appreciate  Narges K. Sobhani for her useful comments and presence in the  earlier discussions. 
We wish to thank an anonymous referee for his/her comments.

\end{document}